\documentclass[11pt]{amsart}

\usepackage{amscd,amsfonts,amsmath,amsthm,anyfontsize,blindtext,centernot,cite,colortbl,enumitem,extarrows,fancyhdr,filecontents,float,footmisc,galois,geometry,hyperref,imakeidx,lastpage,lipsum,mathtools,mathrsfs,MnSymbol,pst-node,spectralsequences,stmaryrd,thmtools,thm-restate,tikz,tikz-cd,titlecaps,titlesec,xcolor}
\usepackage[utf8]{inputenc}
\usepackage[T1]{fontenc}
\usetikzlibrary{positioning}
\makeindex[columns=1, title=Index]
\hypersetup{
    colorlinks,
    linkcolor={red!50!black},
    citecolor={blue!70!black},
    urlcolor={green!80!black}}

\newcommand{\aut}{\textrm{Aut}}
\newcommand{\id}{\text{id}}

\newcommand{\mcg}{\text{MCG}}
 
 \newcommand{\Z}{\mathbb{Z}}
\newcommand{\N}{\mathbb{N}}

\newcommand{\per}{\mathcal{P}}
\newcommand{\F}{\mathscr{F}}

\newtheorem{theorem}{Theorem}[section]

\newtheorem{definition}[theorem]{Definition}
\newtheorem{corollary}[theorem]{Corollary}
\newtheorem{lemma}[theorem]{Lemma}

\theoremstyle{remark}
\newtheorem{remark}[theorem]{Remark}

\titleformat{\section}{\normalfont\normalsize\bfseries}{\thesection}{1em}{}
\titleformat{\subsection}{\normalfont\normalsize}{\thesubsection}{1em}{}
\titleformat{\subsubsection}{\normalfont\normalsize}{\thesubsubsection}{1em}{}
\makeatletter
\def\blfootnote{\xdef\@thefnmark{}\@footnotetext}
\makeatother

\author{Danyu Zhang}
\email{zhang.8939@osu.edu}

\begin{document}

\noindent\begin{tabular}{@{}p{\linewidth}@{}}
  \centering\LARGE Anosov Diffeomorphisms on a Product of Surfaces\\
  \centering\large Danyu Zhang\\
\end{tabular}

\renewcommand{\arraystretch}{1.5} 

\setcounter{tocdepth}{1}

\begin{abstract}
   We show that there is no transitive Anosov diffeomorphism with the global product structure, which is homotopic to a product of pseudo-Anosov diffeomorphisms, on a product of two closed surfaces each of which has genus greater than or equal to two.
\end{abstract}
\makeatletter
\@setabstract
\makeatother

\section{Introduction}

\noindent Let $f:M\to M$ be a $C^1$ diffeomorphism. If there exist an invariant splitting $TM=E^s\oplus E^u$ under $df$, a Riemannian metric $\|\cdot\|$ on $M$, and constants $C>0,\lambda\in (0,1)$, such that
$$\|df^nv^s\|\leq C\lambda^n\|v^s\|,\ \ \ \text{for } v^s\in E^s,$$
$$\|df^{-n}v^u\|\leq C\lambda^n\|v^u\|,\ \ \ \text{for } v^u\in E^u,$$
then we call $f$ an Anosov diffeomorphism.

Smale has aksed \cite{smale-problems} if all Anosov diffeomorphisms live on up to a finite cover a nilmanifold. Yano \cite{yano} has showed that there are no transitive Anosov diffeomorphisms on negatively curved manifolds. Gogolev and Lafont \cite{gog-laf} proved that a product $M_1\times \cdots\times M_k\times N$ where $M_i,\ i=1,\ldots,k$ is a closed negatively curved manifold of dimension $\geq 3$ and $N$ is a nilmanifold does not admit transitive Anosov diffeomorphisms. Neofytidis \cite{neof-4mflds} showed that any $4$-manifolds which are not finitely covered by a product of closed surfaces $S_g\times S_h$ of genus $g,h$ where $g\geq 2$ \emph{or} $h\geq 2$ do not admit transitive Anosov diffeomorphisms.

Consider the pair of foliations of an Anosov diffeomorphism in the universal cover, if each pair of the stable and unstable leaves intersects only once, then we say that the Anosov diffeomorphism has the \emph{global product structure}.
Hammerlindl \cite{hamm} showed that an Anosov diffeomorphism with the polynomial global product structure is topologically conjugate to an infranilmanifold automorphism. The polynomial global product structure implies that the fundamental group has polynomial growth.

Here, we want to partially answer the question asked in \cite{gog-laf}, whether there is an Anosov diffeomorphism on a product of two hyperbolic surfaces.

\begin{theorem}\label{maintheorem} There is no transitive Anosov diffeomorphism with the global product structure on a product of two closed hyperbolic surfaces, which is homotopic to a product of pseudo-Anosov diffeomorphisms.
\end{theorem}

We want to proceed by contradiction. Let $S_1$, $S_2$ denote the two hyperbolic surfaces, and $M:= S_1\times S_2$. Suppose $G:M\to M$ is a transitive Anosov diffeomorphism with the global product structure.

We follow the following steps.

\begin{itemize}
  \item[1.] Reduce $G$ to a product of self-diffeomorphisms $f_1, f_2$ of each surfaces such that
  $$G_\#=(f_1)_\#\times (f_2)_\#:\pi_1(S_1)\times \pi_1(S_2)\to \pi_1(S_1)\times \pi_1(S_2).$$
  Note that then $G$ is homotopic to $f_1\times f_2=: F$.
  \item[2.] Establish Handel \cite{handel_1985} for $f_1\times f_2$ assuming $f_1,f_2$ are pseudo-Anosov, i.e., we show that there exists a closed $G$-invairant subset $Y$ and a continuous surjective map $\varphi:Y\to M$, such that $\varphi\circ G|_Y=F\circ \varphi$.
  \item[3.] Show that $\varphi$ is a homeomorphism by adjusting the argument of Handel \cite{handel_1985} for the Anosov diffeomorphism.
\end{itemize}

\noindent\textbf{Acknowledgement.} I would like to thank Jean Lafont for pointing me to this question and many conversations.

\section{Proof of Theorem \ref{maintheorem}}

\subsection{Reduction of the map}

Denote $\Gamma_1:=\pi_1(S_1)$, $\Gamma_2:=\pi_1(S_2)$, and $\Gamma:=\Gamma_1\times\Gamma_2$.

\begin{lemma} Let $\psi\in\aut(\Gamma)$. Then $\psi^2=\psi_1\times \psi_2$, where $\psi_i\in\aut(\Gamma_i),\ i=1,2$.
\end{lemma}
\begin{proof} Let $(g,h)\in\Gamma$. The centralizer $C_\Gamma((g,h))=C_{\Gamma_1}(g)\times C_{\Gamma_2}(h)$. In a surface group $G$ (of a surface of genus $\geq 2$), if $G \ni x\neq\id$, then $C_G(x)=\Z$ (\cite{FarbMarg}, pg. 23); if $x=\id$, $C_G(x)=G$.

For any $\psi\in\aut(\Gamma)$, we have \begin{tikzcd}[column sep=scriptsize] C_\Gamma(x)\rar["\psi"',"\sim"] & C_\Gamma(\psi(x)) \end{tikzcd}. 

Now if $g=\id$ but $h\neq\id$, $C_\Gamma((g,h))\simeq \Gamma_1\times\Z\simeq C_\Gamma(\psi(g,h))$. So $\psi(g,h)=(\id, h')$ where $h'\neq\id$. This means $\psi(\langle\id\rangle\times\Gamma_2)=\langle\id\rangle\times\Gamma_2$ or $\Gamma_1\times \langle\id\rangle$ if $\Gamma_1\simeq \Gamma_2$. Similarly $\psi(\Gamma_1\times \langle\id\rangle)=\Gamma_1\times \langle\id\rangle$ or $\langle\id\rangle\times\Gamma_2$.
\end{proof}

Again let $G:M\to M$ where $M=S_1\times S_2$, a product of hyperbolic surfaces, be any homeomorphism. Then there exist $f_1\in\mcg(S_1)$ and $f_2\in\mcg(S_2)$ such that $G\simeq f_1\times f_2$.

\subsection{Handel for product of surfaces}

Suppose $G:M\to M$ is a transitive Anosov diffeomorphism and it is homotopic to $f_1\times f_2$ where $f_i:S_i\to S_i,\ i=1,2$ are surface diffeomorphisms. It is easy to see that $f_1$ and $f_2$ cannot be both periodic as mapping classes, by the existence of the Ruelle-Sullivan classes \cite{ruellesullivan} or the growth rate of the periodic points \cite{smale}, but it is not very clear if both of the factors have to be pseudo-Anosov.

We have the following theorem.

\begin{theorem}\label{main-handel} Let $f_1$, $f_2$ be two pseudo-Anosov diffeomorphisms of closed hyperbolic surfaces. Suppose $G: M\to M$ is any map that is homotopic to $F:=f_1\times f_2$ where $M$ is the product of the two surfaces. Then there exists a closed subset $Y\subseteq M$ and a continuous surjective map $\varphi: Y\to M$, homotopic to the inclusion, such that $\varphi\circ G|_Y=F\circ \varphi$.
\end{theorem}

We want to point out that the proof is directly generalized from \cite{handel_1985}. We include the detailed proof for completion in Appendix \ref{proofofhandel}.

\subsection{The contradiction}

Now in this subsection, we assume our Anosov diffeomorphsim $G$ has the following properties:
\begin{itemize}
  \item[(a)] $G$ is transitive, i.e., $\text{NW}(G)=M$, which implies periodic points are dense;
  \item[(b)] $G$ has the global product structure.
\end{itemize}
Furthermore, since $G$ is homotopic to $F=f_1\times f_2$, the product of two pseudo-Anosov diffeomorphisms, $G_\#$ does not fix any free homotopy class by property (2) of Appendix \ref{proofofhandel} neither.

\begin{remark} Note that an Anosov diffeomorphism which satisfies the above properties simulates the behavior of $F$, so we expect that we could run the argument of Handel and get an inverse of the $\varphi$ from Theorem \ref{main-handel}.
\end{remark}

\begin{lemma}[Lemma \ref{period} for Anosov]\label{periodicforanosov} (i) If $y_1,y_2$ are distinct fixed points of $G^n$ then $(G^n,y_1)$ and $(G^n,y_2)$ are not Nielsen equivalent. (ii) If $y$ is $G$-periodic with least period $n$, then there exists $x$ which is $F$-periodic with least period $n$ and such that $(F^n,x)$ is Nielsen equivalent to $(G^n,y)$.
\end{lemma}
\begin{proof} Suppose $G^n$ fixes $y_1$ and $y_2$ and $(G^n,y_1)$ and $(G^n,y_2)$ are Nielsen equivalent, i.e., there exists lifts $\tilde y_1$, $\tilde y_2$ and $t\in\pi_1(M)$ such that $\widetilde G^n\tilde y_1=t\tilde y_1$ and $\widetilde G^n\tilde y_2=t\tilde y_2$. Then $t^{-1}\widetilde G^n$ fixes both $\tilde y_1$ and $\tilde y_2$. Suppose $\tilde z$ is the intersection of the unstable leaf of $\tilde y_1$ and stable leaf $\tilde y_2$. Then $\tilde z$ must also be fixed, which is impossible.

Now the fixed point index of $G$ can only be $\pm 1$. By Theorem 3, pg 94 of \cite{brown}, there exists a fixed point $x$ of $F^n$ such that $(F^n,x)$ is Nielsen equivalent to $(G^n,y)$. We can check that $n$ is the least period of $x$ , by the same argument as in the proof Lemma \ref{period}.
\end{proof}

\begin{remark}\label{bound} Suppose that a pair of foliations $\F_1,\F_2$ in $\widetilde M$ has both the global product structure and the local product structure. Then for any compact set $K\subseteq M$, there exists a constant $B_K$ such that
$$\sup_{x\in K}\{d_i(x,y):\ y\text{ is in the same leaf as $x$ of } \F_i\},\ \ \ i=1,2,$$
where $d_i$ denotes the distance between two points in the same leaf along the leaf of $\F_i$.
This is implicit in the proof of Theorem 1 of \cite{franks-tori}, or see Lemma 5.6 of \cite{my} for a detailed fact checking.
\end{remark}

\begin{corollary}[Theorem \ref{theorem1} (ii) for Anosov]\label{existenceshadow} For all $y\in M$, there exists an $x\in M$ such that $(F,x)\sim (G,y)$; if $y$ is $G$-periodic with least period $n$, then $x$ can be chosen to be $F$-periodic with least period $n$.
\end{corollary}
\begin{proof} Lemmas \ref{nielsenequiv} and \ref{uniformbound} are still true because we are working with $F$. Note that since Nielsen equivalence and $K$-global shadowing are symmetric, because they are equivalence relations, for our fixed $F,G$, we can simply apply the statements for $F$.

Now because the periodic points of $G$ are also dense, for any $y\in M$, there exists sequence of periodic points $y_n$ that approaches $y$ and each of which is globally shadowed by a periodic point $x_n$ of $F$, by Lemma \ref{periodicforanosov}. Then we can choose a convergent subsequence of $x_n$. The limit point globally shadows $y$, by Lemma \ref{uniformbound}. It has least period $n$ follows the same argument as in Lemma \ref{period} (see also Remark \ref{nielglob}).
\end{proof}

\begin{theorem} The $\varphi$ defined in Theorem \ref{main-handel} is a homeomorphism.
\end{theorem}
\begin{proof} Recall that in the proof of Theorem \ref{main-handel} (Theorem \ref{handelprodsurfaces}), we have defined
$$Y=\{y\in M:\ \text{there exists } x\in M \text{ which globally shadows } y\},$$
and $\varphi(y)=x$ where $x$ globally shadows $y$.
By Corollary \ref{existenceshadow}, $Y=M$. We only need to show that $\varphi$ is injective.

Assume there are $y_1\neq y_2$ such that $\varphi(y_1)=\varphi(y_2)$. Then there are lifts $\tilde y_1\neq\tilde y_2$ such that $\tilde\varphi(\tilde y_1)=\tilde\varphi(\tilde y_2)$. 

Consider the unique intersection $\tilde z$ of the unstable leaf $u_1$ of $\tilde y_1$ and the stable leaf $s_2$ of $\tilde y_2$. Let $d_u$ and $d_s$ denote the distances between two points along the unstable and stable leaves respectively. We know that
\begin{align}\label{iterate}
d_u(\widetilde G^k\tilde y_1,\widetilde G^k\tilde z)\geq C\mu^{-k} d_u(\tilde y_1,\tilde z),\ \ \text{and}\\
d_s(\widetilde G^{-k}\tilde y_2,\widetilde G^{-k}\tilde z)\geq C\mu^{-k} d_s(\tilde y_2,\tilde z),
\end{align}
where $\mu\in (0,1)$ and $k>0$. 

From the topology of the measured foliations of the pseudo-Anosov maps, we know that each pair of stable and unstable leaves also intersect only once. Also, $\tilde\varphi$ maps local stable and unstable discs of $\widetilde G$ to the local stable and unstable discs of $\widetilde F$ respectively. Thus $\tilde\varphi (\tilde z)=\tilde\varphi(\tilde y_1)=\tilde\varphi(\tilde y_2)$. In addition, any iterates of $\tilde y_1,\tilde y_2$ under $\widetilde G$ are also mapped to the same point, because $\tilde\varphi(\widetilde G^n\tilde y_i)=\widetilde F^n(\tilde \varphi\tilde y_i), i=1,2$, similarly for $\tilde z$.

On the other hand if we let $K$ denote the preimage of a fundamental domain. It is compact because $\varphi$ is proper. There is an upperbound $B_K>0$ depending on $K$ such that
$$\sup_{x\in K}\big\{d_u(x,y):\ y\in u(x;\widetilde G)\big\}\leq B_K.$$
(See Remark \ref{bound}.) But now (\ref{iterate}) tells us that the distance between $\widetilde {G}^n\tilde y_1$ and $\widetilde G^n\tilde z$ along the leaf becomes unbounded. A contradiction.
\end{proof}


\appendix

\section{Proof of Handel for a product of surfaces}\label{proofofhandel}

Let $f_1:S_1\rightarrow S_1$, $f_2:S_2\rightarrow S_2$ be two pseudo-Anosov diffeomorphisms of closed surfaces $S_1$, $S_2$, denote $M:=S_1\times S_2$, and let $G:M\rightarrow M$ be any map that is homotopic to $F:=f_1\times f_2$.

We follow \cite{handel_1985} to prove the following theorem.

\begin{theorem}[Theorem 2 of \cite{handel_1985}]\label{handelprodsurfaces} There exists a closed $G$-invariant subset $Y\subseteq M$ and a continuous surjective map $\varphi:Y\to M$ that is homotopic to an inclusion, such that $\varphi\circ G|_Y=F\circ \varphi$.
\end{theorem}

We have the following properties for a pseudo-Anosov homeomorphism $f:S\to S$, where $S$ is a closed surface.
\begin{itemize}
  \item[(1)] The periodic points of $f$ are dense;
  \item[(2)] The action induced by $f$ on the free homotopy classes of $S$ has no periodic orbits;
  \item[(3)] The fixed point index of a fixed point $x$ of $f^n$ is never $0$;
  \item[(4)] There exist $\lambda>1$ and an equivariant metric $\widetilde{D}$ on the universal cover $\widetilde{S}$ of $S$ such that $\widetilde{D}=\sqrt{\widetilde{D}_s^2+\widetilde{D}_u^2}$, where $\widetilde{D}_s:\widetilde{S}\times\widetilde{S}\rightarrow[0,\infty)$ and $\tilde{D}_u:\widetilde{S}\times\widetilde{S}\rightarrow[0,\infty)$ are equivariant functions satisfying:
  $$\widetilde{D}_u(\tilde{f}\tilde{x}_1,\tilde{f}\tilde{x}_2)=\lambda\widetilde{D}_u(\tilde{x}_1,\tilde{x}_2)\ \ \ \text{and}\ \ \ \widetilde{D}_s(\tilde{f}^{-1}\tilde{x}_1,\tilde{f}^{-1}\tilde{x}_2)=\lambda\widetilde{D}_s(\tilde{x}_1,\tilde{x}_2)$$
  for all $\tilde{x}_1,\tilde{x}_2\in\widetilde{S}$ and all lifts $\tilde{f}$ of $f$.
\end{itemize}

We now check the corresponding properties for $F$.

\begin{itemize} \item[(1')] The periodic points of $F$ are dense.
\end{itemize}
\begin{proof} Let $\per_1,\per_2,\per$ be sets of periodic points of $f_1,f_2$ and $F$, respectively. Then if $(x_1,x_2)\in\per_1\times\per_2$, there exist $n_1,n_2$ such that $f^{n_1}x_1=x_1$ and $f^{n_2}x_2=x_2$, so $F^{N}(x_1,x_2)=(x_1,x_2)$ for some $N$ and $\per_1\times\per_2\subseteq\per$. Since the periodic points of $f_1$ and $f_2$ are dense, $\overline{\per}_1=S_1$ and $\overline{\per}_2=S_2$. For any $(x,y)\in M$, any neighborhood $U\times V$ that contains $(x,y)$, there exist $x'\in U\cap\per_1, y'\in V\cap\per_2$, so $(x',y')\in U\times V\cap\per_1\times\per_2\subseteq\per$. Therefore $\overline{\per}\supset M$ and $\overline{\per}=M$.
\end{proof}

\begin{itemize} \item[(2')] The action induced by $F$ on the free homotopy classes of $M$ has no periodic orbits.
\end{itemize}
\begin{proof} We know that $\pi_1(M)=\pi_1(S_1)\times\pi_1(S_2)$ and $F_\#=(f_1)_\#\times (f_2)_\#$. Suppose $F_\#$ has a periodic orbit. There exist $\alpha\in \pi_1(S_1)$, $\beta\in\pi_1(S_2)$, and $n\in\N$ such that $\big((f_1)^n_\#\alpha,(f_2)^n_\#\beta\big)=(\alpha,\beta)$. Then $\alpha$ is periodic for $(f_1)_\#$ and $\beta$ periodic for $(f_2)_\#$, which contradicts to (2).
\end{proof}

\begin{itemize} \item[(3')] The fixed point index of a fixed point $(x,y)\in M$ of $F^n$ is never $0$.
\end{itemize}

\begin{proof} By \cite{brown}, Theorem 6, p. 60., the index of a product is the product of the indices.
\end{proof}

\begin{itemize} \item[(4')] Let $\widetilde{D}_1=\sqrt{\widetilde{D}_{1s}^2+\widetilde{D}_{1u}^2}, \widetilde{D}_2=\sqrt{\widetilde{D}_{2s}^2+\widetilde{D}_{2u}^2}$ denote the equivariant metrics on the universal covers $\widetilde{S}_1, \widetilde{S}_2$ respectively, which satisfy (4) above.

For $\tilde x= (\tilde{x}_1,\tilde{x}_2),\ \tilde y=(\tilde{y}_1,\tilde{y}_2)\in \widetilde{M}$, define the product metric on $\widetilde{M}$
$$\widetilde{D}_{s}\big((\tilde{x}_1,\tilde x_2),(\tilde{y}_1,\tilde{y}_2)\big)=\sqrt{\widetilde{D}_{1s}^2(\tilde{x}_1,\tilde{y}_1)+\widetilde{D}_{2s}^2(\tilde{x}_2,\tilde{y}_2)},$$
$$\widetilde{D}_{u}\big((\tilde{x}_1,\tilde{x}_2),(\tilde{y}_1,\tilde{y}_2)\big)=\sqrt{\widetilde{D}_{1u}^2(\tilde{x}_1,\tilde{y}_1)+\widetilde{D}_{2u}^2(\tilde{x}_2,\tilde{y}_2)},$$
and $\widetilde{D}=\sqrt{\widetilde{D}_s^2+\widetilde{D}_u^2}:\widetilde{M}\times\widetilde{M}\rightarrow[0,\infty)$. Then $\widetilde{D}$ is an equivariant metric on $\widetilde{M}$ because if $\alpha\in\pi_1(S_1)$ and $\beta\in\pi_1(S_2)$,
\begin{align*}
\widetilde{D}\big((\alpha,\beta)(\tilde{x}_1,\tilde{x}_2),(\alpha,\beta)(\tilde{y}_1,\tilde{y}_2)\big)&= \sqrt{\widetilde{D}_1^2(\alpha\tilde{x}_1,\alpha\tilde{y}_1)+\widetilde{D}_2^2(\beta\tilde{x}_2,\beta\tilde{y}_2)}\\
&=\sqrt{\widetilde{D}_1^2(\tilde{x}_1,\tilde{y}_1)+\widetilde{D}_2^2(\tilde{x}_2,\tilde{y}_2)}\\
&=\widetilde{D}\big((\tilde{x}_1,\tilde{x}_2),(\tilde{y}_1,\tilde{y}_2)\big).
\end{align*}

Suppose $\lambda_1$ and $\lambda_2$ are the constants from (4) such that 
$$\widetilde{D}_{1u}(\tilde{f}_1\tilde{x}_1,\tilde{f}_1\tilde{y}_1)=\lambda_1\widetilde{D}_{1u}(\tilde{x}_1,\tilde{y}_1),\ \ \text{and}\ \ \widetilde{D}_{2u}(\tilde{f}_2\tilde{x}_2,\tilde{f}_2\tilde{y}_2)=\lambda_2\widetilde{D}_{2u}(\tilde{x}_2,\tilde{y}_2).$$
Let $\lambda=\min\{\lambda_1,\lambda_2\}>1, \lambda'=\max\{\lambda_1,\lambda_2\}>1$. We have
\begin{align*}
\lambda\widetilde{D}_{u}\big((\tilde{x}_1,\tilde{x}_2),(\tilde{y}_1,\tilde{y}_2)\big)&\leq\widetilde{D}_u\big((\tilde{f}_1\times\tilde{f}_2)(\tilde{x}_1,\tilde{x}_2),(\tilde{f}_1\times\tilde{f}_2)(\tilde{y}_1,\tilde{y}_2)\big)\\
&=\sqrt{\widetilde{D}_{1u}^2(\tilde{f}_1\tilde{x}_1,\tilde{f}_1\tilde{y}_1)+\widetilde{D}_{2u}^2(\tilde{f}_2\tilde{x}_2,\tilde{f}_2\tilde{y}_2)}\\
&=\sqrt{\lambda_1^2\widetilde{D}_{1u}^2(\tilde{x}_1,\tilde{y}_1)+\lambda_2^2\widetilde{D}_{2u}^2(\tilde{x}_2,\tilde{y}_2)}\\
&\leq \lambda'\widetilde{D}_{u}\big((\tilde{x}_1,\tilde{x}_2),(\tilde{y}_1,\tilde{y}_2)\big).
\end{align*}
Similarly,
$\lambda\widetilde{D}_{s}(\tilde x,\tilde y)\leq\widetilde{D}_s(\widetilde{F}^{-1}\tilde x,\widetilde{F}^{-1}\tilde y) \leq\lambda'\widetilde{D}_{s}(\tilde x,\tilde y)$. \qed
\end{itemize}

Let $p:\widetilde{M}\to M$ denote the covering projection. We fix a lift $\widetilde{F}=\tilde{f}_1\times\tilde{f}_2:\widetilde{M}\to\widetilde{M}$ of $F$. Then there is a unique lift $\widetilde{G}:\widetilde{M}\rightarrow\widetilde{M}$, that is equivariantly homotopic to $\widetilde{F}$.

\begin{definition} The $f$-orbit of $x$ is $K$-globally shadowed by the $g$-orbit of $y$ if there are lifts $\tilde{x}$ of $x$ and $\tilde{y}$ of $y$ such that $\widetilde{D}(\tilde{f}^k\tilde{x},\tilde{g}^k\tilde{y})\leq K$ for all $k\in\Z$. We write $(f,x)\sim^K(g,y)$ or $(f,x)\sim(g,y)$ if the shadowing constant $K$ is not specified.
\end{definition}

\begin{definition} If $x$ is a fixed point of $f^n$ and $\tilde{x}$ is a lift of $x$, then $\tilde{f}^n\tilde{x}=s\tilde{x}$ for some covering translation $s$ of $\widetilde{M}$. Similarly, if $y$ is a fixed point of $g^n$ and $\tilde{y}$ is a lift of $y$, then $\tilde{g}^n\tilde{y}=t\tilde{y}$ for some covering translation $t$. We say that $(f^n,x)$ and $(g^n,y)$ are Nielsen equivalent if there exist $\tilde{x}$ and $\tilde{y}$ such that $s=t$.
\end{definition}

\begin{remark} The definition of Nielsen equivalence above is equivalent to ``$H$-related'' \cite{brown}, i.e., if $f$ is homotopic to $g$ by the homotopy $H:X\times I\to X$, and $x,y$ are fixed points of $f$ and $g$ respectively, there exists a path $C:I\to X$ such that $H(C(t),t)$ is homotopic to $C(t)$ relative to $x,y$.

Indeed, if there exists lifts $\tilde x, \tilde y$ such that $\tilde f^n\tilde x=s\tilde x$ and $\tilde g^n\tilde y=s\tilde y$ where $s\in\pi_1(X)$, then by path-connectedness of $\widetilde X$, take any path $C$ connecting $\tilde x,\tilde y$, $\widetilde H(C(t),t)$ is homotopic to $sC(t)$ which is just another lift of the same path, because $\widetilde X$ is simply-connected. Thus the projection of $C$ to $X$ is a path that we want.

Conversely, if there exists a path $C$ in $X$ such that $C(0)=f^nx=x$, $C(1)=g^ny=y$, and $H(C(t),t)$ is homotopic to $C$where $H$ is a homotopy between $f$ and $g$, then the lift of the homotopy between $H(C(t),t)$ and $C$ tells us that $\tilde f^n\tilde x=s\widetilde C(0)$ and $\tilde g^n\tilde y=s\widetilde C(1)$ because the fibres are descrete.
\end{remark}

For periodic points $x$ of $F$ and $y$ of $G$, both of period $n$, the fact that the orbit of $x$ is $K$-globally shadowed by the orbit of $y$ is equivalent to that $(F^n,x)$ and $(G^n,y)$ being Nielsen equivalent.

\begin{lemma}[Lemma 1.7 of \cite{handel_1985}]\label{nielsenequiv} If $x$ is a fixed point of $F^n$ and $y$ is a fixed point of $G^n$, then $(F^n,x)$ is Nielsen equivalent to $(G^n,y)$ if and only if $(F,x)\sim(G,y)$.
\end{lemma}
\begin{proof} First suppose $(F^n,x)$ and $(G^n,y)$ are Nielsen equivalent, so there exist lifts $\tilde{x},\tilde{y}\in\widetilde{M}$ and a covering transformation $t$ such that $\widetilde{F}\tilde{x}=t\tilde{x}$ and $\widetilde{G}\tilde{y}=t\tilde{y}$. Then
\begin{align*}
\widetilde{D}(\widetilde{F}^k\tilde{x},\widetilde{G}^k\tilde{y})&=\widetilde{D}(\widetilde{F}^{k-n}\widetilde{F}^n\tilde{x},\widetilde{G}^{k-n}\widetilde{G}^n\tilde{y})\\
&=\widetilde{D}(\widetilde{F}^{k-n}t\tilde{x},\widetilde{G}^{k-n}t\tilde{y})\\
&=\widetilde{D}(\widetilde{F}^{k-n}t\widetilde{F}^{-(k-n)}\widetilde{F}^{k-n}\tilde{x},\widetilde{G}^{k-n}t\widetilde{G}^{-(k-n)}\widetilde{G}^{k-n}\tilde{y})\\
&=\widetilde{D}(t'\widetilde{F}^{k-n}\tilde{x},t'\widetilde{G}^{k-n}\tilde{y})\\
&=\widetilde{D}(\widetilde{F}^{k-n}\tilde{x},\widetilde{G}^{k-n}\tilde{y}),
\end{align*}
where $t'=\widetilde{F}^{k-n}t\widetilde{F}^{-(k-n)}=\widetilde{G}^{k-n}t\widetilde{G}^{-(k-n)}$ is another covering transformation, equal because $F\simeq G$. Thus $\widetilde{D}(\widetilde{F}^k\tilde{x},\widetilde{G}^k\tilde{y})$ takes on only finitely many values, namely, for $k=0,1,\ldots,n-1$, and it is bounded.

Conversely, if $(F,x)\sim(G,y)$, then there exist lifts $\tilde{x}$ of $x$, and $\tilde{y}$ of $y$ such that $\widetilde{D}(\widetilde{F}^k\tilde{x},\widetilde{G}^k\tilde{y})\leq K$ for some $K>0$ and for all $k\in\Z$. Suppose $\widetilde{F}^n\tilde{x}=s\tilde{x}$ and $\widetilde{G}^n\tilde{y}=t\tilde{y}$. Then $s^{-1}\widetilde{F}^n\tilde{x}=\tilde{x}$. Since $F\simeq G$, for any $\gamma\in\pi_1(M)$, $G_\#\gamma=F_\#\gamma$. Thus
$$\widetilde{D}(\tilde{x},(s^{-1}\widetilde{G}^n)^k\tilde{y})=\widetilde{D}((s^{-1}\widetilde{F}^n)^k\tilde{x},(s^{-1}\widetilde{G}^n)^k\tilde{y})=\widetilde{D}(\widetilde{F}^{nk}\tilde{x},\widetilde{G}^{nk}\tilde{y})\leq K,\ \ \text{for all}\ k\in\Z.$$

Any bounded subset of $\widetilde{M}$ intersects only finitely many lifts of $y$, and $(s^{-1}\widetilde{G}^n)^k\tilde{y}$ is just another lift of $y$. There exists an $N\in \N_{\geq 0}$ such that $(s^{-1}\widetilde{G}^n)^N\tilde{y}=\tilde{y}$. On the other hand $$(s^{-1}\widetilde{G}^n)^{N+1}\tilde{y}=(s^{-1}\widetilde{G}^n)\tilde{y}=s^{-1}t\tilde{y}=(s^{-1}\widetilde{G}^n)^N (s^{-1}t\tilde{y}).$$
So $$s^{-1}t(s^{-1}\widetilde{G}^n)^k\tilde{y}=s^{-1}t\tilde{y}=(s^{-1}\widetilde{G}^n)^ks^{-1}t\tilde{y}.$$
This implies that $s^{-1}t(s^{-1}\widetilde{G}^n)^k=(s^{-1}\widetilde{G}^n)^ks^{-1}t$, so $G^n_\#(s^{-1}t)=s^{-1}t$. Then by (2'), $s^{-1}t=1$, so $s=t$.
\end{proof}

\begin{remark} Note that only the second half of the proof relies on the fact that $F$ is a product of pseudo-Anosov diffeomorphisms, more precisely, poperty (2'). Thus we can claim that Nielsen equivalence always implies $K$-global shadowing.
\end{remark}

\begin{lemma}[Lemma 2.1 of \cite{handel_1985}]\label{period} (i) If $x_1$ and $x_2$ are distinct fixed points of $F^n$, then $(F^n, x_1)$ and $(F^n, x_2)$ are not Nielsen equivalent; (ii) If $x$ is $F$-periodic with least period $n$, then there exists $y$ which is $G$-periodic with least period $n$ and such that $(F^n,x)$ is Nielsen equivalent to $(G^n,y)$.
\end{lemma}
\begin{proof} For (i), suppose $(F^n, x_1)$ and $(F^n, x_2)$ are Nielsen equivalent. There are lifts $\tilde{x}_1$ of $x_1$ and $\tilde{x}_2$ of $x_2$ such that $\widetilde{F}\tilde{x}_1=t\tilde{x}_1$ and $\widetilde{F}\tilde{x}_2=t\tilde{x}_2$. Then $t^{-1}\widetilde{F}\tilde{x}_1=\tilde{x}_1$ and $t^{-1}\widetilde{F}\tilde{x}_2=\tilde{x}_2$. $t^{-1}\widetilde{F}$ has to be a lift of $F$ that fixes both $\tilde{x}_1$ and $\tilde{x}_2$, but by (4'), there is no lift of any iterate of $F$ that can fix two distinct points. 

Now we prove (ii). Let $H:M\times I\rightarrow M$ be the homotopy such that $H(x,0)=F^n(x)$ and $H(x,1)=G^n(x)$. By (3'), the fixed point index of any fixed point of $F^n$ is never zero, so Theorem 3, p. 94 in \cite{brown} states that there exists a fixed point of $y$ that is Nielsen equivalent to $x$.

It is sufficient to show that $y$ has least period $n$.

Now fix lifts $\tilde{x},\tilde{y}$ such that $\widetilde{F}^n\tilde{x}=t\tilde{x}$ and $\widetilde{G}^n\tilde{y}=t\tilde{y}$, $t\in\pi_1(M)$, so $t^{-1}\widetilde{F}^n\tilde{x}=\tilde{x}$ and $t^{-1}\widetilde{G}^n\tilde{y}=\tilde{y}$, which implies that $t^{-1}\widetilde{G}^n$ is a lift of $G^n$ that fixes $\tilde{y}$. Suppose $y$ has least period $m_1<n$ and let $m_2=n/m_1>1$. There exist unique lifts $\widetilde{F}^n$ and $\widetilde G^n$ such that $\widetilde{F}^n\tilde{x}=\tilde{x}$ and $\widetilde G^n\tilde y=\tilde y$. We can find a $t_1\in\pi_1(M)$ such that $t^{-1}\widetilde{G}^n=(t_1\widetilde{G}^{m_1})^{m_2}$. Since $t_1\widetilde{F}^{m_1}$ is equivariantly homotopic to $t_1\widetilde{G}^{m_1}$ and by uniqueness of the lift, $t^{-1}\widetilde{F}^n=(t_1\widetilde{F}^{m_1})^{m_2}$. This implies, for any $k\in\Z$,
$$(t^{-1}\widetilde{F}^n)(t_1\widetilde{F}^{m_1})^k\tilde{x}=(t_1\widetilde{F}^{m_1})^{m_2+k}\tilde{x}=(t_1\widetilde{F}^{m_1})^kt^{-1}\widetilde{F}^n\tilde{x}=(t_1\widetilde{F}^{m_1})^k\tilde{x},$$
that is, $t^{-1}\widetilde{F}^n$ fixes the entire $t_1\widetilde{F}^{m_1}$ orbit of $\tilde{x}$. But we observed by (4') that no lift of an iterate of $F$ can fix two distinct points, so $t_1\widetilde{F}^{m_1}$ fixes $\tilde{x}$ and is equal to $t^{-1}\widetilde{F}^n$. Therefore $m_1=n$.
\end{proof}

\begin{remark} Here we need properties (3') and (4').
\end{remark}

The shadowing constant is independent of the chosen points by the following lemma.

\begin{lemma}[Lemma 2.2 of \cite{handel_1985}]\label{uniformbound} There exists $K=K(G)$ (dependent on $G$) such that $(F,x)\sim (G,y)$ if and only if $(F,x)\sim^K (G,y)$. In particular, if $x_n\to x,y_n\to y$, and $(F,x_n)\sim (G,y_n)$ then $(F,x)\sim (G,y)$.
\end{lemma}
\begin{proof} First suppose $(F,x)\sim^K (G,y)$. There exists a shadowing constant $K$, such that the $F$-orbit of $x$ is $K$-globally shadowed by the $G$-orbit of $y$. Then by definition this means $(F,x)\sim (G,y)$.

Next suppose $(F,x)\sim (G,y)$. Let 
$$R=\max\bigg\{\sup_{\widetilde{x}\in\widetilde{M}}\widetilde{D}(\widetilde{F}\tilde{x},\widetilde{G}\tilde{x}),\ \sup_{\tilde{x}\in\widetilde{M}}\widetilde{D}(\widetilde{F}^{-1}\tilde{x},\widetilde{G}^{-1}\tilde{x})\bigg\}.$$
This maximum is reached and $R<\infty$ because $F$ is homotopic to $G$, $M$ is compact and the metric is equivariant.

(4') then implies that
$$\widetilde{D}_u(\widetilde{F}\tilde{x},\widetilde{G}\tilde{y})\geq\widetilde{D}_u(\widetilde{F}\tilde{x},\widetilde{F}\tilde{y})-\widetilde{D}_u(\widetilde{F}\tilde{y},\widetilde{G}\tilde{y})\geq\lambda\widetilde{D}_u(\tilde{x},\tilde{y})-R,$$
and similarly $\widetilde{D}_s(\widetilde{F}^{-1}\tilde{x},\widetilde{G}^{-1}\tilde{y})\geq\lambda\widetilde{D}_s(\tilde{x},\tilde{y})-R$.

Let $K=2(R+1)/(\lambda-1)$. If $\tilde{D}_u(\tilde{x},\tilde{y})>K/2$, then
$$\widetilde{D}_u(\widetilde{F}\tilde{x},\widetilde{G}\tilde{y})-\widetilde{D}_u(\tilde{x},\tilde{y})\geq(\lambda-1)\widetilde{D}_u(\tilde{x},\tilde{y})-R>1,$$
so $\widetilde{D}_u(\widetilde{F}\tilde{x},\widetilde{G}\tilde{y})>1+\widetilde{D}_u(\tilde{x},\tilde{y})$. If $\widetilde{D}_s(\tilde{x},\tilde{y})>K/2$, then $\widetilde{D}_s(\widetilde{F}^{-1}\tilde{x},\widetilde{G}^{-1}\tilde{y})>1+\widetilde{D}_s(\tilde{x},\tilde{y})$. This means if any of the distance between the iterates of $\tilde x,\tilde y$ exceeds $K/2$, the orbits cannot globally shadow each other, for any constant $K'$.

Therefore, if $(F,x)\sim (G,y)$, there must be lifts $\tilde{x}, \tilde{y}$ such that $\widetilde{D}_u(\widetilde{F}^k\tilde{x},\widetilde{G}^k\tilde{y})\leq K/2$ and $\widetilde{D}_s(\widetilde{F}^{k}\tilde{x},\widetilde{G}^{k}\tilde{y})\leq K/2$ for all $k\in\Z$. Then $\widetilde{D}(\widetilde{F}^{k}\tilde{x},\widetilde{G}^{k}\tilde{y})\leq K$ for all $k\in\Z$. We have found a uniform bound, namely $K$, for the shadowing constant. We can write $(F,x)\sim^K (G,y)$. This $K$ is independent of $x,y$.

Now we prove the second statement of the lemma. Suppose $x_n\to x,y_n\to y$, and for each $n$, fix lifts so $\widetilde{D}(\widetilde{F}^{k}\tilde{x}_n,\widetilde{G}^{k}\tilde{y}_n)\leq K$ for all $k$. We claim that there exist convergent subsequences $\{\tilde{x}_{n_j}\}$ and $\{\tilde{y}_{n_j}\}$ that converge to lifts $\tilde{x},\tilde y$ so $\widetilde D(\widetilde F^k\tilde x,\widetilde G^k\tilde y)\leq K$ for all $k$.

Take $U\subseteq M$, a neighborhood of $x$ that is evenly covered by the covering projection $p$. There is an $N\in\N$ such that for all $n\geq N$, $x_n\in U$. Take $\widetilde{U}\subseteq p^{-1}U$ to be the connected component that contains $\tilde{x}_N$, the lift that we picked as above. Then for each $\tilde{x}_n\in\{\tilde{x}_n\}_{n>N}$ there is a $t_n\in\pi_1(M)$ such that $t_n\tilde{x}_n\in\widetilde{U}$, so $\{t_n\tilde{x}_n\}_{n>N}$ converges to a lift $\tilde{x}$ of $x$. We want to show that $\{t_n\tilde{y}_n\}_{n>N}$ contains a convergent subsequence. For any $n>N$,
$$\widetilde{D}(t_{N+1}\tilde{y}_{N+1},t_n\tilde{y}_n)\leq\widetilde{D}(t_{N+1}\tilde{y}_{N+1},t_{N+1}\tilde{x}_{N+1})+\widetilde{D}(t_{N+1}\tilde{x}_{N+1},t_n\tilde{x}_n)+\widetilde{D}(t_n\tilde{x}_n,t_n\tilde{y}_n)\leq 2K+\varepsilon,$$
for a small $\varepsilon$, the diameter of $\widetilde{U}$, by the uniform boundedness we showed above and equivariance of the metric. Thus $\{t_n\tilde{y}_n\}_{n>N}$ is a bounded sequence. Therefore, there is a subsequence $\{t_{n_j}\tilde{y}_{n_j}\}$ that converges to some lift $\tilde{y}$ of $y$.

Now $$\widetilde{D}(\widetilde{F}^{k}\tilde{x},\widetilde{G}^{k}\tilde{y})\leq\widetilde{D}(\widetilde{F}^{k}\tilde{x},\widetilde{F}^{k}\tilde{x}_{n_j})+\widetilde{D}(\widetilde{F}^{k}\tilde{x}_{n_j},\widetilde{G}^{k}\tilde{y}_{n_j})+\widetilde{D}(\widetilde{G}^{k}\tilde{y}_{n_j},\widetilde{G}^{k}\tilde{y})\leq K+2\varepsilon,$$
for large enough $n_j$ and small $\varepsilon$. Thus $(F,x)\sim (G,y)$.
\end{proof}

\begin{remark} Here we need (4').
\end{remark}

\begin{theorem}[Theorem 1 of \cite{handel_1985}]\label{theorem1} (i) $(F,x_1)\sim (F,x_2)$ implies that $x_1=x_2$; (ii) For all $x\in M$, there exists $y\in M$ such that $(F,x)\sim (G,y)$; if $x$ is $F$-periodic with least period $n$, then $y$ can be chosen to be $G$-periodic with least period $n$.
\end{theorem}
\begin{proof} By (4') we have (i).

By (1') periodic points of $F$ are dense in $M$. For any $x\in M$, there exists a sequence $\{x_n\}$ of periodic points of $F$, each with least period $p_n$, such that $x_n\to x$. Then by Lemma \ref{period} (ii), for each $n$, there exists $y_n$ that is $G$-periodic with least period $p_n$ and $(F^{p_n},x_n)$ is Nielsen equivalent to $(G^{p_n},y_n)$. By Lemma \ref{nielsenequiv}, $(F,x_n)\sim(G,y_n)$ for all $n$. Since $M$ is compact, there is a subsequence $\{y_{n_k}\}\subseteq\{y_n\}$ such that $y_{n_k}\to y$ as $k\to\infty$, for some $y\in M$, so $(F, x_{n_k})\sim(G,y_{n_k})$ and $x_{n_k}\to x,\ y_{n_k}\to y$. By Lemma \ref{uniformbound}, we have $(F,x)\sim(G,y)$.
\end{proof}

\begin{remark}\label{nielglob} Note here we used the fact that Nielsen equivalence implies global shadowing to find points that shadow, which does not require property (2'). We need the other direction of the statement of Lemma \ref{nielsenequiv} to show that the point we found has the same least period. (4') causes more problem to generalize to reducible mapping classes.
\end{remark}

\begin{proof}[Proof of \ref{handelprodsurfaces}] Let $Y=\{y\in M: \exists x\in M\text{ such that } (F,x)\sim (G,y)\}$. For any $x_0\in M$, by Theorem \ref{theorem1} (ii), there exsits $y_0\in M$ such that $(F,x)\sim (G,y)$. Thus we can define a surjective map $\varphi:Y\rightarrow M$ by $\varphi(y_0)=x_0$. It is well-defined by Theorem \ref{theorem1} (i).

Next we show that $Y$ is closed. take a convergent sequence $\{y_n\}\subseteq Y$ such that $y_n\rightarrow y$ for some $y\in M$. For each $y_n$ there is $x_n$ such that $(F,x_n)\sim(G,y_n)$. Because $M$ is compact, there is a subsequence $\{x_{n_k}\}$ of $\{x_n\}$ such that $x_{n_k}\rightarrow x\in M$. Thus by Lemma \ref{uniformbound}, we have $(F,x)\sim (G,y)$ and so $y\in Y$ and $Y$ is closed.

By the above and how we define $\varphi$, if $\{y_n\}\subseteq Y$ is such a sequence that $y_n\rightarrow y\in Y$, then we have $x_n=\varphi(y_n)$ and $\varphi(y)=x$. Since $M$ is compact and every convergent subsequence converges to the same limit $x$, because $x_{n_k}\rightarrow x'\neq x$ implies $(F,x')\sim (G,y)$, $\{x_n\}$ also converges to $x$.

We define $\tilde{\varphi}(\tilde{y})=\tilde{x}$ for such lifts so we have $\widetilde{D}(\tilde{f}^k\tilde{x},\tilde{g}^k\tilde{y})\leq K$ for all $k\in\Z$. Then $\widetilde{D}(\tilde{\varphi}(\tilde{y}),\tilde{y})\leq K$ for $\tilde{y}\in\tilde{Y}$, so $\varphi$ is homotopic to the inclusion.

Finally, take $y\in Y$. Since $(F,x)\sim (G,y)$ implies $(F,F(x))\sim (G,G(y))$, $F\circ\varphi(y)=F(x)=\varphi(G(y))$. Therefore $F\circ\varphi=\varphi\circ G|_Y$.
\end{proof}


\nocite{*}
\bibliographystyle{halpha}
\bibliography{bibfile}



\end{document}